\documentclass{amsart}

\usepackage{amssymb,graphicx}

\newtheorem{theorem}{Theorem}[section]

\newtheorem{lemma}[theorem]{Lemma}

\theoremstyle{remark}

\newtheorem{property}[theorem]{Property}

\newcommand{\cgB}{\mathcal{B}}

\newcommand{\cgF}{\mathcal{F}}

\newcommand{\cgL}{\mathcal{L}}

\newcommand{\cgR}{\mathcal{R}}

\newcommand{\Inc}{\operatorname{Inc}}

\newcommand{\bdim}{\operatorname{bdim}}
\newcommand{\ldim}{\operatorname{ldim}}

\newcommand{\ple}{\operatorname{ple}}
\newcommand{\cut}{\operatorname{Cut}}

\begin{document}
\title[COMPONENTS AND BLOCKS]{Boolean Dimension, Components and Blocks}

\author[M\'{E}SZ\'{A}ROS]{Tam\'{a}s M\'{e}sz\'{a}ros}
\address{Tam\'{a}s M\'{e}sz\'{a}ros \\
Fachbereich Mathematik und Informatik\\
  Kombinatorik und Graphentheorie\\
  Freie Universit\"at Berlin\\
  Germany}
\email{meszaros.tamas@fu-berlin.de}

\author[MICEK]{Piotr Micek}
\address{Piotr Micek\\
Theoretical Computer Science Department\\
  Faculty of Mathematics and Computer Science\\
  Jagiellonian University\\
  Krak\'ow, Poland}
\email{piotr.micek@tcs.uj.edu.pl}

\author[TROTTER]{William T. Trotter}
\address{William T. Trotter\\
School of Mathematics\\
  Georgia Institute of Technology\\
  Atlanta, Georgia}

\email{trotter@math.gatech.edu}



\thanks{Tam\'as M\'esz\'aros is supported by the Dahlem Research School of Freie Universit\"at Berlin.}
\thanks{Piotr Micek is partially supported by a Polish National Science Center grant (SONATA BIS 5; UMO-2015/18/E/ST6/00299).}
\thanks{William~T. Trotter is supported by a Simons Foundation Collaboration Grant.}

\begin{abstract}
We investigate the behavior of Boolean dimension with respect to
components and blocks.  To put our results in context, we note that
for Dushnik-Miller dimension, we have that if $\dim(C)\le d$ for
every component $C$ of a poset $P$, then $\dim(P)\le \max\{2,d\}$; also if $\dim(B)\le d$ for every block $B$ of a poset
$P$, then $\dim(P)\le d+2$. By way of constrast, 
local dimension is well behaved with respect to components, but 
not for blocks: if $\ldim(C)\le d$ for every component $C$ of a 
poset $P$, then $\ldim(P)\le d+2$; however, for every $d\ge4$, there
exists a poset $P$ with $\ldim(P)=d$ and  $\dim(B)\le 3$ for every block 
$B$ of $P$. In this paper we show that Boolean dimension behaves 
like Dushnik-Miller dimension with respect to both 
components and blocks: if $\bdim(C)\le d$ for every component $C$ of
$P$, then $\bdim(P)\le 2+d+4\cdot2^d$; also if $\bdim(B)\le d$ for every block of $P$, then $\bdim(P)\le 
19+d+18\cdot 2^d$.
\end{abstract}

\maketitle

\section{Notation and Terminology}

We consider combinatorial problems for
finite posets. As has become standard in the literature, we
use the terms \textit{elements} and \textit{points} interchangeably
in referring to the members of the ground set of some poset $P=(X,<)$.
We will write $x\parallel y$ when $x$ and $y$ are \emph{incomparable}
in $P$, and we let $\Inc(P)$ denote the set of all ordered pairs 
$(x,y)$ with $x\parallel y$ in $P$.  As a binary 
relation, $\Inc(P)$ is symmetric. The
\textit{dual} of a poset $P$ will be denoted by $P^*$, while the dual of a linear order $L$ on $X$ by $L^*$. If $L$ is a linear order on $X$ and $Y\subseteq X$, then we will write $L(Y)$ for the restriction of $L$ to $Y$. 
We will also use the notation $L=[A<B]$ when the elements of $X$ can
be labeled so that $L=[u_1<u_2<\cdots<u_m]$ and $A=[u_1<u_2<\cdots<u_{k}]$, $B=[u_{k+1}<u_{k+2}<\cdots<u_m]$ for some index $k$.
This notation then generalizes naturally to an expression such as
$L=[A_1<A_2<\cdots<A_s]$. For two elements $x,y\in P$ we say that $x$ \emph{covers} $y$ if $y<x$ and there is no element $z\in P$ with $y<z<x$. The \emph{cover} graph of $P$ has vertex set the elements of $P$, and two vertices $x$ and $y$ are joined by an edge if one of them covers the other in $P$.  Finally, we will also use the now standard notation $[n]=\{1,2,\dots,n\}$.

\medskip

A nonempty family $\cgR=\{L_1,L_2,\dots,L_d\}$ of linear extensions
of $P$ is called a \textit{realizer} of $P$ when 
$x\leq y$ in $P$ if and only if $x\leq y$ in $L_i$ for each $i=1,2,\dots,d$.
Clearly, $\cgR$ is a realizer if and only if for each $(x,y)\in
\Inc(P)$, there is some $i$ for which $x>y$ in $L_i$.
The \textit{dimension} of a poset $P$, as defined by Dushnik
and Miller in their seminal paper~\cite{bib:DusMil}, is the least
positive integer $d$ for which $P$ has a realizer of size~$d$. A subset $S$ of $\Inc(P)$ is called \emph{reversible} if there is a linear extension $L$ of $P$ with $x>y$ in $L$ for every $(x,y)$. When $P$ is not a chain then the dimension of $P$ is the least positive integer $d$ for which there is a partition $\Inc(P)=S_1\cup S_2 \cup \dots \cup S_d$ with each $S_i$ reversible. A subset $\{(x_\alpha,y_\alpha)\ :\ \alpha\in [k]\}$ of $\Inc(P)$ is called an \emph{alternating cycle} if $x_\alpha\leq y_{\alpha+1}$ for all $\alpha\in [k]$, with addition on the indicies understood cyclically. It is easy to see that alternating cycles are not reversible. On the other hand, Trotter and Moore~\cite{bib:TroMoo} proved that they are the only obstructions for a set of incomparable pairs to be reversible, i.e., $S\subseteq \Inc(P)$ is reversible if and only if it does not contain an alternating cycle. For more details about now standard concepts and techniques for working with Dushnik-Miller
dimension, the reader may consult any of several recent research
papers, e.g.,~\cite{bib:JMMTWW}, \cite{bib:StrTro} and \cite{bib:TrWaWa}
or the
research monograph~\cite{bib:Trot-Book}.

\medskip

In recent years, researchers have been investigating combinatorial
problems for two variations of Dushnik-Miller dimension,
known as \textit{Boolean dimension} and \textit{local dimension},
respectively.

\medskip

%For a positive integer $d$, we let $\bftwo^d$ denote the set of all
%$0$--$1$ strings of length~$d$. Such strings are
%also called \textit{bit strings}.  
Let $P$ be a poset with at least two elements and
let $\cgB=\{L_1,L_2,\dots,L_d\}$ be a nonempty family of linear orders (need not to be linear extensions of $P$) on the
ground set of $P$.
Also, let $\tau$ be a Boolean function which maps all $0$-$1$ strings of length~$d$
to $\{0,1\}$.  For each ordered pair $(x,y)$ of distinct elements of $P$,
we form the bit string $q(x,y,\cgB)$ of length~$d$
which has value $1$ in coordinate $i$
if and only if $x<y$ in $L_i$.  We call the pair $(\cgB,\tau)$ a
\textit{Boolean realizer} of $P$ if for every ordered pair $(x,y)$ of distinct
elements of $P$, we have $x<y$ in $P$ if and only if $\tau(q(x,y,\cgB))=1$.
Ne\v{s}et\v{r}il and Pudl\'{a}k~\cite{bib:NesPud} (by slightly modifying the definition of Gambosi, Ne\v{s}et\v{r}il and Talamo~\cite{bib:GNT})
defined the \textit{Boolean dimension of $P$},
denoted $\bdim(P)$, as the least positive integer $d$ for which
$P$ has a Boolean realizer $(\cgB,\tau)$ with
$|\cgB|=d$. By convention, the Boolean dimension of a one element poset is $1$.

Clearly, $\bdim(P)\le\dim(P)$, since if $\cgR=\{L_1,L_2,\dots,
L_d\}$ is a realizer of $P$, we can simply take $\tau$ to be the function which
maps the all ones bit string $(1,\dots,1)$ to itself while all other bit strings of length~$d$
are mapped to~$0$. Trivially, $\bdim(P)=1$ if and only if $P$ is either a chain or an
antichain; if $Q$ is a subposet of $P$, then
$\bdim(Q)\le\bdim(P)$; and $\bdim(P)=\bdim(P^*)$.
It is an easy exercise to show that if $\bdim(P)=2$, then
$\dim(P)=2$, while Trotter and Walczak \cite{bib:TroWal} proved the
modestly more challenging fact that if $\bdim(P)=3$, then
$\dim(P)=3$.  

\medskip

Again, let $P$ be a poset.  A \textit{partial linear extension},
abbreviated \textit{ple}, of $P$ is a linear extension of a subposet of $P$.
Whenever $\cgL$ is a family of $\ple$'s of $P$  and
$u\in P$, we set $\mu(u,\cgL)=|\{L\in\cgL:u\in L\}|$.  In turn, we
set $\mu(P,\cgL)=\max\{\mu(u,\cgL):u\in P\}$.  A family $\cgL$
of $\ple$'s of a poset $P$ is called a
\textit{local realizer} of $P$ if for every pair $(x,y)$ of distinct
elements of the ground set of $P$, we have $x<y$ in $P$ unless there
is some $L\in\cgL$ with $x>y$ in $L$.  The \textit{local dimension of 
$P$}, denoted $\ldim(P)$, is then defined\footnote{The concept of
local dimension is due to Torsten Ueckerdt \cite{bib:Ueck} and
was shared with participants of the workshop \textit{Order and Geometry}
held in Gu\l towy, Poland, September 14--17, 2016. Ueckerdt's 
new concept resonated with participants at the workshop and
served to rekindle interest in the notion of Boolean dimension as well.}
to be the least positive integer $d$ for which $P$ has
a local realizer $\cgL$ with $\mu(P,\cgL)=d$.

Clearly, $\ldim(P)\le\dim(P)$ since every realizer $P$ is also a local realizer. It is again easily seen that $\ldim(P)=1$ if and
only if $P$ is a chain; if $Q$ is a subposet
of $P$, then $\ldim(Q)\le\ldim(P)$; and $\ldim(P^*)=\ldim(P)$.
It is an easy exercise to show that if $\ldim(P)=2$, then $\dim(P)=2$.

\medskip

Recall that for $n\ge2$, the \textit{standard example} $S_n$ 
is a height~$2$ poset with minimal elements $A=\{a_1,a_2,\dots,a_n\}$, maximal elements $B=\{b_1,b_2,\dots,b_n\}$ and $a_i<b_j$ in $S_n$ if and
only if $i\neq j$. As is well known, $\dim(S_n)=n$.  On the other hand, it is
another easy exercise to show that $\bdim(S_n)=4$ for all $n\geq 4$, while $\ldim(S_n)=3$ for all $n\geq 3$.

\section{Statements of Results}

To state our main results, we will need some basic concepts of
graph theory, including connected and disconnected
graphs, components, cut vertices, and $k$-connected graphs for an
integer $k\ge2$. Recall that when $G$ is a graph, a connected
induced subgraph $H$ of $G$ is called a \emph{block} of $G$ when
$H$ is maximal connected subgraph without a cut vertex, or quivalently it is either a maximal $2$-connected subgraph, a bridge (with its ends) or an isolated vertex. Now a poset $P$ is said to
be \emph{connected} if its cover graph is connected.  A subposet $B$ of $P$ is
said to be \emph{convex} if $x,z\in B$ and $x<y<z$ in $P$ imply $y\in B$. When $B$ is a convex subposet of $P$, the
cover graph of $B$ is an induced subgraph of the cover graph of $P$.
A convex subposet $B$ of $P$ is called a \emph{component} of $P$ when the
cover graph of $B$ is a component of the cover graph of $P$.
A convex subposet $B$ of $P$ is called a \emph{block} of $P$ when
the cover graph of $B$ is a block of the cover graph of $P$.
A point $x$ in a poset $P$ is called a \emph{cut vertex} of $P$ when
$x$ is a cut vertex of the cover graph of $P$.

As is well known, when $P$ is a disconnected poset with components
$C_1$, $C_2$, $\dots$, $C_t$, then 
\begin{equation*}
\dim(P)=\max\{2,\max\{\dim(C_i):1\le i\le t\}\}.
\end{equation*}
For local dimension, it is an easy exercise to show that
\begin{equation*}
\ldim(P)\le 2+\max\{\ldim(C_i):1\le i\le t\},
\end{equation*}
but we do not know whether this inequality is best possible.  

We prove a corresponding, but somewhat more complicated, result for 
Boolean dimension. This is the first of our two main results.

\begin{theorem}\label{thm:bdim-components}
Let $P$ be a disconnected poset with components
$C_1,C_2,\dots,C_t$.  If $d=\max\{\bdim(C_i):1\le i\le t\}$,
then $\bdim(P)\le2 + d+ 4\cdot 2^d$.
\end{theorem}
 
The situation with blocks is more
complex. For Dushnik-Miller dimension, Trotter, Walczak and Wang~\cite{bib:TrWaWa} proved that
when $P$ is a connected poset with blocks
$B_1$, $B_2$, $\dots$, $B_t$, then
\begin{equation*}
\dim(P)\leq 2+\max\{\dim(B_i):1\le i\le t\}.
\end{equation*}
Furthermore, this inequality is best possible. Neither the proof of the inequality,
nor the proof that the inequality is best possible is elementary.
Surprisingly, however, there is no parallel result for local dimension,
as Bosek, Grytczuk and Trotter~\cite{bib:BoGrTr} proved that for
every $d\ge4$, there is a poset $P$ with $\ldim(P)\ge d$,
such that $\ldim(B)\le 3$ whenever $B$ is a block in $P$.

The second of our two main results is the following theorem
showing that Boolean dimension behaves like Dushik-Miller dimension
and not like local dimension when it comes to blocks, i.e., we show
that the Boolean dimension of a poset is bounded in terms
of the maximum Boolean dimension among its blocks.

\begin{theorem}\label{thm:bdim-blocks}
Let $P$ be a poset with blocks
$B_1,B_2,\dots,B_t$.  If $d=\max\{\bdim(B_i):1\le i\le t\}$, then $\bdim(P)\le 19+d+18\cdot 2^d$. 
\end{theorem}

%Although Theorem~\ref{thm:bdim-blocks} is stated for
%connected posets, it will be clear that an analogous
%result holds for disconnected posets, with a revised
%upper bound of $19+d+18\cdot2^d$.

We doubt that the inequalities in Theorem~\ref{thm:bdim-components} and Theorem~\ref{thm:bdim-blocks}
are sharp, but in some sense they are not so far from the truth, as we
will show that for large $d$, there is a disconnected poset $P$ with
$\bdim(P)=2^{\Omega(d)}$ and $\bdim(C)\le d$ for every
component $C$ of $P$ (and hence $\bdim(B)\leq d$ for every block $B$ of $P$).

\section{Proofs}

In discussing Boolean realizers for a poset $P$,
the phrase ``a pair $(x,y)$'' will always refer to an ordered pair
of distinct elements of $P$. Trivially, a poset $P$ has Boolean dimension at most $d$, if $P$ has
a Boolean realizer $(\cgB,\tau)$ with $|\cgB|=d$. In defining a Boolean realizer, most of the work will
go into the construction of the linear orders in the family $\cgB$. Typically,
$\cgB$ will be made up of subfamilies of linear orders, each subfamily serving to reveal certain 
details concerning a pair $(x,y)$. As far as the Boolean formula $\tau$ is considered, rather than explicitly writing out the
rule for the function $\tau$, we will simply explain how we can determine 
whether $x$ is less than $y$ in $P$ based on the bits associated with
the linear orders in $\cgB$.  As we shall see, there are times when
we know whether $x<y$ in $P$ after seeing just a few of the
bits in $q(x,y,\cgB)$, while in other instances, we may need to see all
or nearly all of the bits.

We will make frequent use of two simple lemmas. Both are standard tools in the field, but we nevertheless include the short proofs, as they
are instructive for the more complex results to follow.

\begin{lemma}\label{lem:partition}
Let $P$ be a poset with ground set $X$, and let $\phi$ be
a $t$-coloring of $X$.
Then there is a family $\cgF=\{N_1,N_2\}$ of
two linear orders on $X$ so that
given a pair $(x,y)$ of distinct elements of $X$, we can determine
whether $\phi(x)$ is the same as $\phi(y)$ from the bits 
associated with the linear orders in $\cgF$.
\end{lemma}

\begin{proof}
We may, without loss of generality, assume $\phi$ uses the integers in $[t]$ as colors.
For each $i\in [t]$, let $X_i$ consist of all $x\in X$
with $\phi(x)=i$, and let $L_i$ be an arbitrary
linear order on $X_i$. Then set
\begin{align*}
N_1=\ &[L_1<L_2<L_3<\dots<L_{t-1}<L_t]\quad \text{and}\\
N_2=\ &[L^*_1<L^*_2<L^*_3<\dots<L^*_{t-1}<L^*_t].
\end{align*}
Now for the family $\cgF=\{N_1,N_2\}$, if $\phi(x)=\phi(y)$, we
will get either $(0,1)$ or $(1,0)$, but if $\phi(x)\neq\phi(y)$,
we will get either $(0,0)$ or $(1,1)$.\qed
\end{proof}

\begin{lemma}\label{lem:color}
Let $P$ be a poset with ground set $X$, and let $\phi$ be
a $t$-coloring of $X$.
Then there is a family $\cgF$ of $4\lceil\log_2 t\rceil$ linear orders
on $X$ so that given a pair $(x,y)$ of distinct elements of $X$, we can determine the pair
$(\phi(x),\phi(y))$ of colors from the bits associated with the
linear orders in $\cgF$.
\end{lemma}

\begin{proof}
Setting $r=\lceil\log_2 t\rceil$, we may, without loss of generality, assume $\phi$ uses the subsets of $[r]$ as colors.
For each $j\in[r]$, let $X_j$ consist of all
$u\in X$ with $j\in \phi(u)$ and let $L_0$ be an arbitrary linear order on $X$.
Then for each $j\in [r]$, we add the following four linear orders
to the family $\cgF$:
\begin{align*}
M_1(j)=\ &[L_0(X_j)<L_0(X\setminus X_j)]\\
M_2(j)=\ &[L_0^*(X_j)<L_0(X\setminus X_j)]\\
M_3(j)=\ &[L_0(X\setminus X_j)< L_0(X_j)]\\
M_4(j)=\ &[L_0(X\setminus X_j)< L_0^*(X_j)]
\end{align*}
If $j\in\phi(x)$ and $j\in\phi(y)$, then the bits for the query will
either be $(1,0,1,0)$ or $(0,1,0,1)$.  If $j\in\phi(x)$ and
$j\not\in\phi(y)$, then the bits will be $(1,1,0,0)$.
Conversely, if $j\notin \phi(x)$ and $j\in\phi(y)$, then
we will have $(0,0,1,1)$. Finally, if $j\not\in\phi(x)$ and
$j\not\in\phi(y)$, then the query will return either
$(1,1,1,1)$ or $(0,0,0,0)$. Hence the $4r$ linear orders
together will clearly enable us to determine the pair $(\phi(x),\phi(y))$.\qed
\end{proof}

\subsection{Boolean Dimension and Components}

In this subsection, we prove Theorem~\ref{thm:bdim-components}.

\begin{proof}
Let $P$ be a disconnected
poset with components $C_1,C_2,\dots,C_t$, and assume that
$\bdim(C_i)\le d$, for each $i\in[t]$. Let $X$ be the ground set of $P$, and for each $i\in[t]$ let
$X_i$ be the ground set of the component $C_i$. Furthermore, for each $i\in [t]$,
let $(\cgB_i,\tau_i)$ be a Boolean realizer of $C_i$ with  $|\cgB_i|=d$. We label the linear orders in $\cgB_i$ as
$\{L_j^i:j\in[d]\}$.

We now show that $P$ has a Boolean realizer $(\cgB,\tau)$, with
$|\cgB|=2+d+4\cdot 2^d$.
The family $\cgB$ will be the union
\begin{equation*}
\cgB=\cgF_1\cup\cgF_2\cup\cgF_3,
\end{equation*}
where
\begin{align*}
|\cgF_1|&=2, & |\cgF_2|&=4\cdot 2^d,  & |\cgF_3|&=d.
\end{align*}

We begin by defining a coloring $\phi_1:X\rightarrow[t]$ by setting 
$\phi_1(x)=i$ when $x\in X_i$.  We use Lemma~\ref{lem:partition} to 
determine a family $\cgF_1$ of size~$2$ such that for each pair $(x,y)$,
the bits for the linear orders in $\cgF_1$
determine whether $\phi_1(x)$ is equal to $\phi_1(y)$.

Next consider the set $\mathbb{T}=\{\tau_i:i\in[t]\}$. Although the integer $t$ is not bounded in
terms of $d$, the size of $\mathbb{T}$ is at most $2^{2^d}$. Therefore, the coloring $\phi_2:X\rightarrow
\mathbb{T}$ defined by setting $\phi_2(x)=\tau_i$ when
$x\in X_i$ uses at most $2^{2^d}$ colors. 
Using Lemma~\ref{lem:color},  we take $\cgF_2$ as a family of
$4\cdot 2^d$ linear orders on $X$ so that given a pair $(x,y)$, 
we can determine the pair $(\phi_2(x),\phi_2(y))$ from the bits associated 
with the linear orders in $\cgF_2$.

Finally for each $j\in[d]$, let $L_j$ be a linear order on $X$ such that
for each $i\in[t]$, the restriction of $L_j$ to $X_i$ is $L_j^i$, and take $\cgF_3=\{L_j:j\in[d]\}$.

Now let $(x,y)$ be a pair.
From the bits associated with the linear orders in
$\cgF_1$, we know whether $x$ and $y$ are in the same component or not. 
If not, then we know that $x$ and $y$ are incomparable in $P$.  So 
we can assume that we have learned that $x$ and $y$ are in the 
same component. Next, from the bits associated with the linear orders in
$\cgF_2$, we can learn the common color $\phi_2(x)=\phi_2(y)$, which
is the truth function $\tau_i$ for the component $C_i$ containing both
$x$ and $y$. Then we can apply the truth function $\tau_i$ to the
bits for the linear orders in $\cgF_3$. Since the restriction of these linear orders to $X_i$ is $\cgB_i$, this will finally answer whether
$x$ is less than $y$ in $C_i$, and hence in $P$. This finishes the proof of Theorem~\ref{thm:bdim-components}.\qed
\end{proof}

Now we explain why the bound in Theorem~\ref{thm:bdim-components} cannot be improved
dramatically.
Consider a large integer $n$ and the family $\mathbb{P}_n$ of all
posets $P$ of height at most~$2$ on the ground set $X=A\cup B$, with all elements of $A=\{a_1,a_2,\dots,a_n\}$ being minimal
in $P$ and all elements of $B=\{b_1,b_2,\dots,b_n\}$ being maximal in $P$.  Clearly, there are $2^{n^2}$ such posets, since
for each pair $(a_i,b_j)\in A\times B$, we can choose whether
or not $a_i<b_j$ in $P$.

In~\cite{bib:NesPud}, Ne\v{s}et\v{r}il and Pudl\'ak show that
if $P$ is a poset on $2n$ points, then $\bdim(P)\le c\log_2 n$ for some universal constant $c$, and
they basically use the family $\mathbb{P}_n$ to show that this inequality
is essentially best possible. This follows from the fact that there are not more then $\left((2n)!\right)^s2^{2^s}$ Boolean realizers with $s$ linear orders, and hence if $\bdim(P)\le s$ for all $P\in\mathbb{P}_n$, then we must have $\left((2n)!\right)^s2^{2^s}\ge |\mathbb{P}_n|=2^{n^2}$. However, this implies that $s=\Omega(\log_2 n)$.  

Now consider the disconnected poset $P$ formed by taking the
disjoint sum of a copy of each poset in $\mathbb{P}_n$.  Setting
$d=c\log_2 n$, we then have $\bdim(C)\le d$ for every component $C$
of $P$.  On the other hand, we claim that 
$\bdim(P)=2^{\Omega(d)}$. To see this, suppose that
$\bdim(P)=m$ and let $(\cgB,\tau)$ be a Boolean realizer for
$P$ with $|\cgB|=m$. Now let $Q$ be any poset from $\mathbb{P}_n$, and
let $\cgB_Q$ be the family of linear orders obtained
by taking the restrictions of the linear orders in $\cgB$ to
the ground set of $Q$.  Then $(\cgB_Q,\tau)$ is a Boolean
realizer for $Q$. Since $\tau$ is now fixed, the number of realizers we can produce in such a way is at most $\left((2n)!\right)^m$. However, then we must have $(2n!)^m\ge |\mathbb{P}_n|=2^{n^2}$, which implies 
$m=\Omega\left(\frac{n}{\log_2 n}\right)=2^{\Omega(d)}$.

\subsection{Boolean Dimension and Blocks}

In this subsection, we prove Theorem~\ref{thm:bdim-blocks}.
To start with, we first describe one of the key ideas, extracted from~\cite{bib:TrWaWa}.
In the argument to follow, we will encounter the following
situation.  We will have a poset $P$ with a cut vertex $w$ and two connected convex 
subposets $Q$ and $Q'$, such that their ground sets, $Y$ and
$Y'$, respectively, share only the element $w$. Then clearly the subposet $Q''$ of $P$
with ground set $Y''=Y\cup Y'$ is connected and convex, and the
point $w$ is a cut vertex of $Q''$. Then if $L=[A<w<B]$ and $L'=[C<w<D]$ are linear orders of $Y$ and $Y'$,
respectively, there are many ways to determine a linear order
$L''$ on $Y''$ such that $L''(Y)=L$ and $L''(Y')=L'$.  However,
in our argument, we will \emph{always} do it using the following \emph{merge rule}:\quad $L''=[A<C<w<D<B]$.  It is important to note
that this choice forces points in $A\cup B$ to the
``outside'' while concentrating points of $C\cup D$ in the
``inside''.

Now on to the proof. First let $P$ be a connected poset with
$\bdim(B)\le d$ for every block $B$ of $P$.  We will build
a Boolean realizer $(\cgB,\tau)$ for $P$ with 
$|\cgB|\le 17+d+18\cdot2^d$.  The family $\cgB$ will be the
union
\begin{equation*}
\cgB=\cgF_1\cup\cgF_2\cup\dots\cup\cgF_{11}
\end{equation*}
where:
\begin{align*}
|\cgF_1|&=|\cgF_5|=|\cgF_6|=2,  &|\cgF_2|&=|\cgF_8|=|\cgF_9|=4\cdot 2^d,  & |\cgF_3|&=d\\
|\cgF_4|&= 3  & |\cgF_7|&=|\cgF_{11}|=4,
\end{align*}
and $\cgF_{10}$ is the union of $2\cdot 2^d$ families, each of size~$3$. As a consequence, we will have $|\cgB|=17+d+18\cdot2^d$, as required.

Let $\mathbb{B}$ denote the set of all blocks of $P$, and let $t=|\mathbb{B}|$.  We may assume $t\ge2$, otherwise
$P$ itself is a block and $\bdim(P)\le d$.
Let $\mathbb{B}=\{B_1,B_2,\dots,B_t\}$ be a labelling of the blocks in 
$P$ so that whenever $2\le i\le t$, block $B_i$ has a 
(necessarily unique) point in common with $B_1\cup B_2\cup\dots\cup B_{i-1}$.
This point will be called the \textit{root} of $B_i$ and denoted $\rho(B_i)$.
The block $B_1$ does not have a root. For each $i\in[t]$, we put $X_i$ for the ground set of $B_i$, and
we let $Y_i=X_1\cup X_2\cup\dots\cup X_i$.  We also set $Z_1=X_1$ and $Z_i=X_i-\{\rho(B_i)\}$ for $2\le i\le t$. Then clearly
$Z_1\cup Z_2\cup\dots\cup Z_t$ is a partition of $X$.

The first part of the proof will closely parallel the argument
for Theorem~\ref{thm:bdim-components}.  As before, we first define a coloring $\phi_1:X\rightarrow[t]$ by setting $\phi_1(u)=i$ when $u\in Z_i$.
Using Lemma~\ref{lem:partition}, we then 
take $\cgF_1$ as a family of two linear orders so that
given a pair $(x,y)$, the bits for the linear orders in $\cgF_1$ determine whether $\phi_1(x)=\phi_1(y)$.

For each $i\in[t]$, let $(\cgB_i,\tau_i)$ be a Boolean realizer
for $B_i$ with $|\cgB_i|=d$. Then we again take the set $\mathbb{T}=\{\tau_i:i\in[t]\}$ which has size at most $2^{2^d}$, and consider the coloring $\phi_2:X\rightarrow
\mathbb{T}$ defined by setting $\phi_2(x)=\tau_i$ when
$x\in Z_i$. Just as before, this is a coloring using at most $2^{2^d}$ colors, 
so using Lemma~\ref{lem:color},  we can take $\cgF_2$ to be the family of
$4\cdot 2^d$ linear orders on $X$ so that given a pair $(x,y)$, 
we can determine the pair $(\phi_2(x),\phi_2(y))$ from the bits associated 
with the linear orders in $\cgF_2$.

Now we label the linear orders
in $\cgB_i$ as $\{L_j^i:j\in[d]\}$.  
Recall that in the proof of Theorem~\ref{thm:bdim-components}, we chose a family 
$\{L_j:j\in[d]\}$ of linear orders on $X$ such that for each 
$(i,j)\in[t]\times[d]$, the restriction of $L_j$ to $X_i$ is 
$L_j^i$.  Here we must be more careful in the construction of 
these linear orders. For each $j\in[d]$, we define a linear order $L_j$ on $X$ using
the following recursive procedure.  First, set $L_j(1)=L_j^1$.
Then suppose that for some $k\in[t-1]$, we have already defined a
linear order $L_j(k)$ on the set $Y_k$.  
Now let $w=\rho(B_{k+1})$. Then $w$ is both in $Y_k$ and $X_k$, so there is a suitable $A$ and $B$ such that $L_j(k)=[A<w<B]$,  
and a suitable $C$ and $D$ such that $L_j^{k+1}=[C<w<D]$.  We now can define $L_j(k+1)$ by the merge rule discussed previously, i.e., we put
$L_j(k+1)= [A<C<w<D<B]$. At the end, when this procedure stops, we take $L_j=L_j(t)$ and set $\cgF_3=\{L_j:j\in[d]\}$. 
Note that for each pair $(i,j)\in[t]\times [d]$, the restriction 
of $L_j$ to $X_i$ is still $L_j^i$.

We summarize what we have accomplished with the families
$\cgF_1$, $\cgF_2$ and $\cgF_3$.  Let $(x,y)$ be a pair.
If there is some $i\in[t]$ such that $x,y\in Z_i\subseteq X_i$, then
this fact will be detected by the linear orders in $\cgF_1$.
The linear orders in $\cgF_2$ will then detect the truth-function
$\tau_i=\phi_2(x)=\phi_2(y)$ of $B_i$. Then, as the restriction of the linear orders in $\cgF_3$ is $\cgB_i$, we can determine whether
$x$ is less than $y$ in $P$ simply by applying $\tau_i$ to the
bits for the linear orders in $\cgF_3$.  As a consequence, for the balance 
of the argument, from now on we restrict our attention to pairs
$(x,y)$ satisfying the following property.

\begin{property}\label{prop:1}
$\phi_1(x)\neq\phi_1(y)$, i.e., there
is no $i\in[t]$ for which $x,y\in Z_i$.
\end{property}

Next we define a digraph, which we will call the \textit{root digraph of} $P$. 
Its vertex set is $X$ and for each $2\le i\le t$,
we have an edge between $\rho(B_i)$ and every
$u\in Z_i$ if $u$ is comparable with $\rho(B_i)$ in $P$.
The edge is directed from $u$ to $\rho(B_i)$ when $u<\rho(B_i)$
in $P$ and it is directed from $\rho(B_i)$ to $u$ when
$\rho(u_i)<u$ in $P$.  Evidently, the root digraph of $P$ is a 
directed forest.

The root digraph determines a poset $Q$ whose ground set
is $X$ and $u$ is covered by
$v$ in $Q$ when there is an edge from $u$ to $v$ in the
root digraph.  Evidently, the poset $Q$ is a ``forest'', i.e.,
there are no cycles in the cover graph of $Q$.  A well known
theorem of Trotter and Moore~\cite{bib:TroMoo} asserts that
the dimension of a poset whose cover graph is a forest is at most~$3$, so we add to $\cgB$
a family $\cgF_4$ which is a realizer of size~$3$ for $Q$.

Clearly, $Q$ is a suborder of $P$, i.e., if $x<y$ (resp. $x>y$) in $Q$, as
detected by the bits for $\cgF_4$ being $(1,1,1)$ (resp. $(0,0,0)$), then
$x<y$ (resp. $x>y$) in $P$, so for the balance of the argument, we restrict
our attention to pairs $(x,y)$ which also satisfy the following.
property

\begin{property}\label{prop:2}
$x\parallel y$ in $Q$, i.e., the bits
for the linear orders in $\cgF_4$ are not $(1,1,1)$ or $(0,0,0)$.
\end{property}

Now let $(x,y)$ be a pair satisfying Property~1 and~2. We first decide on the relative position of $x$ and $y$ in the cover graph $G$ of $P$. For this observe that there is a natural tree structure $T$ on the $Z_i$'s, with $Z_{i}$ being a neighbour of $Z_{j}$ when $\rho(B_{i})\in Z_{j}$ or $\rho(B_{j})\in Z_{i}$. We consider $T$ as rooted at $Z_1$ and also fix a planar upward drawing of $T$. Then for every vertex of $T$ there is a natural left-to-right ordering of its children. Let $i_x\neq i_y$ be such that $x\in Z_{i_x}$ and $y\in Z_{i_y}$. We can distinguish four cases:
\begin{itemize}
\item ``$x$ is below $y$" if $Z_{i_x}$ is on the path from $Z_{i_y}$ to $Z_1$ in $T$,
\item ``$y$ is below $x$" if $Z_{i_y}$ is on the path from $Z_{i_x}$ to $Z_1$ in $T$,
\item ``$x$ is left of $y$" if there is some $Z_i$ which has two children $Z_j$ and $Z_k$, with $Z_j$ left of $Z_k$, such that $Z_{i_x}$ is in the subtree rooted at $Z_j$ and $Z_{i_y}$ is in the subtree rooted at $Z_k$, and
\item ``$y$ is left of $x$" which is defined analogously.
\end{itemize}
To identify which case are we in, let $[Z_{i_1},Z_{i_2},\dots,Z_{i_t}]$ be the left-to-right and $[Z_{j_1},Z_{j_2},\dots,Z_{j_t}]$ the right-to-left depth-first search order of the $Z_i$'s according to $T$. For evety $i\in [t]$ choose an arbitrary linear order $M_i$ on the elements of $Z_i$ and put 
\[\cgF_5=\Big\{[M_{i_1}<M_{i_2}<\cdots < M_{i_t}],[M_{j_1}<M_{j_2}<\cdots < M_{j_t}]\Big\}.\]
Then, looking at the bits associated to the linear orders in $\cgF_5$, we see $(1,1)$ exactly if $x$ is below $y$, $(0,0)$ exactly if $y$ is below $x$, $(1,0)$ exactly if $x$ is left of $y$ and $(0,1)$ exactly if $y$ is left of $x$. In what follows, we handle all four cases separately, but, before doing so, we introduce some further notation. 

For each $i\in [t]$ and each $u\in X_i$, we 
define the \textit{tail of $u$ from $X_i$}, denoted
$T(u,X_i)$, as the set of all points
$v\in X$ with the property that every path in the cover graph $G$ of $P$ starting
at $v$ and ending at a point of $X_1$ contains $u$.  Note
that, in particular, $u\in T(u,X_i)$ and $T(u,X_i)$ is a set of consecutive elements in each linear order $L_j\in\cgF_3$.

We put $\cut(x,y)$ for the set 
of cut vertices in the cover graph $G$ of $P$ which are on every path from $x$ to $y$ in $G$. Since there is no $i\in[t]$ for which $x,y\in Z_i$, $\cut(x,y)$ is clearly nonempty.
Let $i=i(x,y)$ denote
the least $j\ge1$ such that $\cut(x,y)
\cap X_j\neq \emptyset$, $u=u(x,y)$ the element of $\cut(x,y)\cap
Z_i$ which is closest to $x$ in $G$, while $v=v(x,y)$ the
element of $\cut(x,y)\cap Z_i$ which is closest to $y$ in $G$.

\medskip

\noindent \textbf{Case ``$x$ below $y$"} Suppose we learned from the family $\cgF_5$ that $x$ is below $y$ and let
\begin{equation*}
S=\Big\{(x,y)\in \Inc(P)\ : \ x\text{ is below }y\text{ and }v=v(x,y)\not< y\text{ in }P\Big\}.
\end{equation*}
We claim that $S$ is reversible. Indeed, suppose to the contrary that $S$ is not reversible, hence it must contain some alternating cycle $\{(x_\alpha,y_\alpha)\ : \ \alpha\in [k]\}$. For $\alpha\in [k]$ we have $x_\alpha\leq y_{\alpha+1}$ and $v_{\alpha+1}=v(x_{\alpha+1},y_{\alpha+1})\not\leq y_{\alpha+1}$, which is possible only if $x_{\alpha+1}$ is below $x_\alpha$. Clearly, this statement cannot hold for all $\alpha\in [k]$ (cyclically).

Let $M$ be the linear extension reversing $S$. We clearly have $x\not<y$ unless the bit for $M$ is $1$, which we assume from now on. In this case $x$ and $v=v(x,y)$ must necessarily be different, otherwise we would get a contradiction either with Property 2 (if $x$ and $y$ are comparable) or with the fact that the bit corresponding to $M$ is $1$ (if $x$ and $y$ are incomparable). As $x$ is below $y$, we have $i_x=i(x,y)$. Then, as $y\in T(v,X_{i_x})$ and $T(v,X_{i_x})$ is a set of consecutive elements (not containing $x$) in each of the $L_i$'s, we have that the relative position of $x$ and $y$ in $L_i$ is always the same as that of $x$ and $v$, i.e., $q(x,v,\cgF_3)=q(x,y,\cgF_3)$. On the other hand, using the family $\cgF_2$ we already learned what $\tau_{i_x}$ is, and we also know that the restriction of $\cgF_3$ to $X_{i_x}$ is $\cgB_{i_x}$. Hence we can apply $\tau_{i_x}$ to the bit string $q(x,y,\cgF_3)=q(x,v,\cgF_3)=q(x,v,\cgB_{i_x})$ to decide whether $x<v$ or not. If not, we clearly also have $x\not< y$. On the other hand, we claim that if we arrive at $x<v$, then this already implies $x<y$. Indeed suppose to the contrary that $x\not< y$. As $M$ is a linear extension and the corresponding bit is $1$, we know that $y\not < x$, hence necessarily $x\parallel y$. However, as $(x,y)$ was not reversed by $M$ we must have $v< y$ and hence $x< v< y$.

\medskip

\noindent \textbf{Case ``$y$ below $x$"} This case can clearly be handled in a symmetric manner involving some other linear extension $M'$. 

\medskip

To cover these two cases we add to $\cgB$ the family $\cgF_6=\{M,M'\}$.

\medskip

\noindent \textbf{Case ``$x$ left of $y$"} Suppose we learned from the family $\cgF_5$ that $x$ is left of $y$. This in particular implies that $x\neq u(x,y)$ and $y\neq v(x,y)$.

Let $I(P)$ denote the set of all pairs $(x,y)$ which satisfy 
Property~1 and~2, with $x$ left of $y$ and $x\parallel y$ in $P$.  Using the
linear order $L_1$ from $\cgF_2$, we define four subsets $R_1,R_2,
R_3,R_4$ of $I(P)$ as follows:

\begin{enumerate}
\item[(1)] $R_1$ consists of all pairs $(x,y)\in I(P)$ such that
$u(x,y) \leq v(x,y)$ in $L_1$ and $x\not< u(x,y)$.
%$y'\in T(u,X_i)$ whenever $x<y'$ in $P$. 
\item[(2)] $R_2$ consists of all pairs $(x,y)\in I(P)$ such that
$u(x,y) \leq v(x,y)$ in $L_1$ and $v(x,y)\not< y$.
%$x'\in T(v,X_i)$ whenever $x'<y$ in $P$.
\item[(3)] $R_3$ consists of all pairs $(x,y)\in I(P)$ such that
$u(x,y) \geq v(x,y)$ in $L_1$ and $x\not< u(x,y)$. 
%$y'\in T(u,X_i)$ whenever $x<y'$ in $P$.
\item[(4)] $R_4$ consists of all pairs $(x,y)\in I(P)$ such that
$u(x,y) \geq v(x,y)$ in $L_1$ and $v(x,y)\not < y$.
%$x'\in T(v,X_i)$ whenever $x'<y$ in $P$.
\end{enumerate}

We claim that each set in $\{R_1, R_2, R_3, R_4\}$ is reversible.  
We give the argument for $R_1$, as
it is clear that the reasoning for the other three cases
is symmetric.  Suppose to the contrary that $R_1$ is not reversible and hence it contains an alternating
cycle $\{(x_\alpha,y_\alpha): \alpha\in[k]\}$.  Let $\alpha\in [k]$ and $i_{\alpha}=i(x_\alpha,y_\alpha)$, $u_{\alpha}=u(x_\alpha,y_\alpha)$,
$v_{\alpha}=v(x_\alpha,y_\alpha)$. Recall that $x_\alpha$ is left of $y_\alpha$ and $x_\alpha\not<u_\alpha$ hence we can have $x_\alpha\le y_{\alpha+1}$ in $P$ only if $y_{\alpha+1}$ is also left of $y_\alpha$. Clearly, this statement cannot hold for all $\alpha\in [k]$ (cyclically).
%by the definition of $R_1$ we know that
%$y_{\alpha+1}\in T(u_\alpha,X_i)$. Furthermore as $u_{\alpha}\neq v_{\alpha}$ we know that $T(u_{\alpha},X_{i_{\alpha}})$ and
%$T(v_{\alpha},X_{i_{\alpha}})$ are disjoint intervals of $L_1$. Therefore $u_{\alpha}<v_{\alpha}$ in $L_1$ implies that all points of $T(u_{\alpha},X_{i_{\alpha}})$ are less than all points of $T(v_{\alpha},X_{i_{\alpha}})$ in $L_1$, in particular, $y_{\alpha+1}<y_\alpha$ in
%$L_1$. Clearly, this statement cannot hold for all $\alpha\in[k]$.

For $j\in [4]$ let then $N_j$ be a linear extensions of $P$ that reverses all pairs
in $R_j$ and set $\cgF_7=\{N_1,N_2,N_3,N_4\}$. Then given a pair $(x,y)$, we conclude that
$x\not< y$ in $P$ unless the bits
for the linear orders in $\cgF_7$ are $(1,1,1,1)$.
So for the balance of the argument, we restrict our attention to
pairs $(x,y)$ which also satisfy the following property:

\begin{property}\label{prop:3}
The bits for the linear orders in $\cgF_7$ are
$(1,1,1,1)$.
\end{property}

Let $(x,y)$ be a pair satisfying Property~1 through~3, 
and let $i=i(x,y)$, $u=u(x,y)$ and $v=v(x,y)$. We claim that the properties enforced on $(x,y)$ imply that $x<u$, $v<y$ in $P$ and $u\neq v$. We give the argument for $x<u$, the reasoning for $v<y$ is clearly symmetric. Suppose to the contrary that $x\not< u$ and hence $x\not< y$. As the linear orders in $\cgF_7$ are linear extensions and the corresponding bits are $(1,1,1,1)$, we know that $y\not< x$, so necessarily we have $x\parallel y$. However then $(x,y)$ is either in $R_1$ or in $R_3$ and hence has to be reversed by either $N_1$ or $N_3$ contradicting Property~3.  Finally to see that $u\neq v$ just note that otherwise we would have $x<u=v<y$ and hence $x<y$ in $Q$ which would contradict Property~2.  

As a consequence we have $x<y$ in $P$ if and only
if $u<v$ in $P$. As $x$ is left of $y$, we clearly have $x\in T(u,X_i)$ and $y\in T(v,X_i)$. Moreover, as $u\neq v$, $T(u,X_i)$ and $T(v,X_i)$ are disjoint intervals in each of the linear orders in $\cgF_3$, therefore $q(x,y,\cgF_3)=q(u,v,\cgF_3)$. Also, by the properties of $\cgF_3$, we have that $q(u,v,\cgF_3)=q(u,v,\cgB_i)$ and hence we can determine whether $u<v$ in $P$ by applying $\tau_i$
to these bits. The rub in these observations is that, in
general, we do not have any apparent method for determining
$\tau_i$.  Accordingly, our goal for the remainder of the argument
is to work around this difficulty.

Given any $a\in X$, we define a uniquely determined
pair $(\sigma_1(a),\sigma_2(a))$ by the following rule.
For $\sigma_1(a)$, we consider sequences of the form
$W=(w_0,w_1,\dots,$ $w_m)$ where 
\begin{enumerate}
\item[(1)] $w_0=a$ and 
\item[(2)] if $0\le j<m$
and $w_j\in Z_i$, then $w_{j+1}=\rho(B_i)$ and $w_j<w_{j+1}$ in $P$.
\end{enumerate}
Among all such
sequences, it is easy to see that there is a largest non-negative
integer $m$ and a uniquely determined element $v\in X$ for
which there is a sequence of this form with $w_m=v$.
We then set $\sigma_1(a)=v$. When $a\neq\sigma_1(a)$,
we have $a<\sigma_1(a)$ in $P$. The definition for $\sigma_2(a)$ is symmetric and when
$a\neq\sigma_2(a)$, we have $a>\sigma_2(a)$.

Now we define two further colorings $\phi_3,\phi_4:X\rightarrow \mathbb{T}$ 
as follows. For $a\in X$ we put $\phi_3(a)=\tau_i$ and $\phi_4(a)=\tau_j$ if $\sigma_1(a)\in Z_i$ and $\sigma_2(a)\in Z_j$. Let $\cgF_8$ and $\cgF_9$ be the families of $4\cdot2^d$ linear orders guaranteed by Lemma~\ref{lem:color} so that given
a pair $(x,y)$, we can determine $(\phi_3(x),\phi_3(y))$ and $(\phi_4(x),\phi_4(y))$ by looking at the bits for the family $\cgF_8$ and $\cgF_9$, respectively.

For each $i\in[t]$, let us fix an arbitrary linear extension $L_0^i$ of 
$B_i$. Then for every subset $\mathbb{S}$ of $\mathbb{T}$,
we define a poset $Q(\mathbb{S})$ with ground set $X$ by describing its cover relations. A point $a$ is covered by a point $b$ in $Q(\mathbb{S})$ if either of 
the following conditions are satisfied:

\begin{enumerate}
\item[(1)] There is $i\in[t]$ such that $a,b\in X_i$, $\tau_i\notin\mathbb{S}$ and the root digraph contains an edge from $a$ to $b$.
\item[(2)] There is $i\in[t]$ such that $a,b\in X_i$, $\tau_i\in\mathbb{S}$ and $a$ is covered by $b$ in $L_0^i$.
\end{enumerate}

\begin{lemma}\label{lem:qs}
For every $\mathbb{S}\subseteq\mathbb{T}$,
$\dim(Q(\mathbb{S}))\le 3$.
\end{lemma}

\begin{proof}
By the result of Trotter and Moore~\cite{bib:TroMoo} mentioned earlier, to prove that $\dim(Q(\mathbb{S}))\le 3$ it is enough to show that the cover graph of $Q(\mathbb{S})$ is a forest. This is easily seen, as for $i\in [t]$, the restriction of the cover graph of $Q(\mathbb{S})$ to $X_i$ is either a star centered at $\rho(B_i)$ (if $\tau_i\notin \mathbb{S}$) or a path of length $|X_i|$ (if $\tau\in \mathbb{S}$). 
%, i.e., every
%block of $Q(\mathbb{S})$ is a $2$-element chain.
%Clearly every cover relation of $Q(\mathbb{S})$ is contained in some
%block of $P$ and hence if $B$ is a block of $Q(\mathbb{S})$,
%then there has to be some $2\le i\le t$ so that the points of
%$B$ belong to $X_i$.  If $\tau_i\not\in\mathbb{S}$, then
%$B$ is a $2$-element chain containing $\rho(B_i)$.  If $\tau_i\in
%\mathbb{S}$, then $B$ is a $2$-element chain formed by two
%consecutive elements of $L_0(B_i)$.
\qed
\end{proof}

A classic result of R\'enyi~\cite{bib:Re} says that given a finite set $A$ there is always a family $\mathcal{A}$ of $\lceil\log_2 |A|\rceil$ subsets of $A$ that separates every pair of elements, i.e., for every distinct $a,b\in A$ there is some set in $\mathcal{A}$ which contains exactly one of them. By adding the complement of every set in $\mathcal{A}$ we arrive at a family $\mathcal{A}'$ of size $2\lceil\log_2 |A|\rceil$ with the property that for every ordered pair of distinct elements $(a,b)\in A^2$ there is some set in $\mathcal{A}'$ which contains $a$ but does not contain $b$. 
By applying this result to $A=\mathbb{T}$, fix a family $\mathcal{S}=\{\mathbb{S}_1,\mathbb{S}_2,\dots,\mathbb{S}_m\}$ of
subsets of $\mathbb{T}$ of size $m=2\cdot 2^d$, so that for every ordered pair
$(\tau_\alpha,\tau_\beta)$ of distinct elements of $\mathbb{T}$,
there is some $j\in[m]$ such that $\tau_\alpha\in\mathbb{S}_j$ and
$\tau_\beta\not\in\mathbb{S}_j$. For each $j\in[m]$ also fix a realizer of 
size~$3$ for the poset $Q(\mathbb{S}_j)$, guaranteed by Lemma~\ref{lem:qs}, and let $\cgF_{10}$ be the union of all these $2\cdot 2^d$ realizers. They enable us to determine the relation of any $x$ and $y$ in $Q(\mathbb{S}_j)$ for each $j\in[m]$.

Now let $(x,y)$ be again a pair satisfying Property~1 through~3 with $i=i(x,y)$, $u=u(x,y)$, $v=v(x,y)$, $\tau_\alpha=\phi_3(x)$
and $\tau_\beta=\phi_4(y)$. Further let $W_x$ and $W_y$ be the sequences witnessing $\sigma_1(x)$ and $\sigma_2(y)$. As $x<u$ and $v<y$, we clearly have that $u$ and $v$ are in the sequences $W_x$ and $W_y$, respectively. Furthermore, we also have that $\rho(B_i)$ is in at most one of the sequences $W_x$ and $W_y$, as otherwise we would have $x<y$ in $Q$, which would contradict Property~2. This clearly implies that at least one of the following statements holds:
\begin{enumerate}
\item[(1)] $\sigma_1(x)=u$ and so $\tau_\alpha=\tau_i$.
\item[(2)] $\sigma_2(y)=v$ and so $\tau_\beta=\tau_i$.
\end{enumerate}

If $\tau_\alpha=\tau_\beta$, then necessarily $\tau_\alpha=\tau_\beta=\tau_i$, so the answer
as to whether $x<y$ in $P$ is given by applying the truth-function
$\tau_i$ to the bits for the linear orders in $\cgF_3$.
Therefore it remains to consider the case where $\tau_\alpha\neq\tau_\beta$.

Using the properties of $\mathcal{S}$, let $j_1$ and $j_2$ be distinct integers in $[m]$ such
that $\tau_\alpha$ belongs to $\mathbb{S}_{j_1}$ but not
to $\mathbb{S}_{j_2}$, while
$\tau_\beta$ belongs to $\mathbb{S}_{j_2}$ but not
to $\mathbb{S}_{j_1}$. Then, by the definition of the posets $Q(\mathbb{S}_{j_1})$ and $Q(\mathbb{S}_{j_2})$, if $\sigma_1(x)=u$ and $\tau_\alpha=\tau_i$ then we have $x\parallel y$
in $Q(\mathbb{S}_{j_2})$, while if $\sigma_2(y)=v$ and $\tau_\beta=\tau_i$ then we have $x\parallel y$
in $Q(\mathbb{S}_{j_1})$. If from $\cgF_{10}$ we learn that $x\parallel y$ both in $Q(\mathbb{S}_{j_1})$ and in $Q(\mathbb{S}_{j_2})$ then $u\not<v$ in the linear extension $L_0^i$ and so we conclude $x\not< y$ in $P$. Therefore, we may assume that $x\parallel y$ in only one of them. In this case this property also identifies in which of the previous two cases we are in, i.e., whether we have  $\tau_i=\tau_\alpha$ or   $\tau_i=\tau_\beta$. Then we may apply this truth function to the bits for the
linear orders in $\cgF_3$ to learn whether $x$ is less than $y$ in $P$.

%Now assume we are in the first case. If $x\not< y$ in $Q(\mathbb{S}_{j_1})$, then
%$u\not< v$ in $L_0^i$ which tells us that $u\not\le v$  in $P$, and hence $x\not\le y$ in $P$.  If $x<y$ in $Q(\mathbb{S}_{j_1})$, then
%the truth-function $\tau_\alpha$ applied to the bits for the
%linear orders in $\cgF_2$ tells us whether $x$ is less than $y$ in $P$.
%
%The argument in the second case is clearly symmetric and as each of the described families has the specified size, with this the proof of Theorem~\ref{thm:bdim-blocks} is complete.

\medskip

\noindent \textbf{Case ``$y$ left of $x$"} This case can clearly be handled in a symmetric manner. As far as the linear orders involved are considered, we may reuse the families $\cgF_8,\cgF_9, \cgF_{10}$ from the previous case, but we need to replace $\cgF_7$ with a new, but analogous family $\cgF_{11}$ of size $4$. 

\medskip

This finishes the description of the families of linear orders and hence the proof of Theorem~\ref{thm:bdim-blocks} for connected posets.

\medskip

To extend the preceding proof to disconnected posets, we simply add at the beginning
two linear orders, guaranteed by Lemma~\ref{lem:partition}, to detect for each pair $(x,y)$ whether $x$ and
$y$ belong to the same component.  Afterwards, we apply the
construction given in the proof to each component.  The manner
in which the linear orders on the components are merged is
arbitrary.

%\medskip
%
%For a lower bound, we return to
%the construction we had for components. Recall that this poset has Boolean dimension $\Omega(2^d/d)$. Each block is a subposet of some component and hence has size at most $2n+1$. Therefore, by the result of Ne\v{s}et\v{r}il and Pudl\'ak, its Boolean dimension is at most $c\log (n+1)\leq d+1$.

%\begin{acknowledgements}
%\end{acknowledgements}



\end{document}